\documentclass[12pt,reqno]{amsart}
\usepackage{latexsym,amssymb,amscd}

\def\B'c{{\mathcal{B'}}}
\def\U'c{{\mathcal{U'}}}

\def\opn#1#2{\def#1{\operatorname{#2}}} 
\opn\chara{char}
\opn\length{\ell}
\opn\projdim{proj\,dim}
\opn\injdim{inj\,dim}
\opn\ini{in}
\opn\rank{rank}
\opn\depth{depth}
\opn\sdepth{sdepth}
\opn\height{ht}
\opn\embdim{emb\,dim}
\opn\codim{codim}

\opn\Tr{Tr}
\opn\bigrank{big\,rank}
\opn\superheight{superheight}\opn\lcm{lcm}
\opn\trdeg{tr\,deg}%
\opn\reg{reg}
\opn\lreg{lreg}
\opn\set{set}
\opn\supp{Supp}
\opn\shad{Shad}
\opn\div{div}
\opn\Div{Div}
\opn\cl{cl}
\opn\Cl{Cl}
\opn\Spec{Spec}
\opn\Supp{Supp}
\opn\supp{supp}
\opn\Sing{Sing}
\opn\Ass{Ass}
\opn\Ann{Ann}
\opn\Rad{Rad}
\opn\Soc{Soc}
\opn\Ker{Ker}
\opn\Coker{Coker}
\opn\Im{Im}
\opn\Hom{Hom}
\opn\Tor{Tor}
\opn\Ext{Ext}
\opn\End{End}
\opn\Aut{Aut}
\opn\id{id}

\opn\nat{nat}
\opn\GL{GL}
\opn\SL{SL}
\opn\mod{mod}
\opn\ord{ord}
\opn\aff{aff}
\opn\con{conv}
\opn\relint{relint}
\opn\st{st}
\opn\lk{lk}
\opn\cn{cn}
\opn\core{core}
\opn\vol{vol}
\opn\gr{gr}

\def\pot#1#2{#1[\kern-0.28ex[#2]\kern-0.28ex]}

\opn\dirlim{\underrightarrow{\lim}}
\opn\invlim{\underleftarrow{\lim}}
\def\pnt{{\raise0.5mm\hbox{\large\bf.}}}

\let\to=\rightarrow

\def\Implies{\ifmmode\Longrightarrow \else
     \unskip${}\Longrightarrow{}$\ignorespaces\fi}
\def\implies{\ifmmode\Rightarrow \else
     \unskip${}\Rightarrow{}$\ignorespaces\fi}
\def\iff{\ifmmode\Longleftrightarrow \else
     \unskip${}\Longleftrightarrow{}$\ignorespaces\fi}

\let\:=\colon
\newtheorem{Theorem}{Theorem}[section]
\newtheorem{Lemma}[Theorem]{Lemma}
\newtheorem{Corollary}[Theorem]{Corollary}
\newtheorem{Proposition}[Theorem]{Proposition}
\theoremstyle{remark}
\newtheorem{Remark}[Theorem]{Remark}
\newtheorem{Method 1}[Theorem]{Method 1}
\newtheorem{first proof}[Theorem]{first proof}

\theoremstyle{definition}
\newtheorem{Definition}[Theorem]{Definition}
\newtheorem{Problem}[Theorem]{Problem}

\let\epsilon=\varepsilon
\let\phi=\varphi
\let\kappa=\varkappa
\numberwithin{equation}{section}

\begin{document}
\title[Matlis Duality]{A Few Comments On Matlis Duality}%
\author{Waqas Mahmood}%
\address{Abdus Salam School of Mathematical Sciences, GCU, Lahore Pakistan}%
\email{waqassms$@$gmail.com}%

\thanks{This research was partially supported by Higher Education Commission, Pakistan}

\begin{abstract}
For a Noetherian local ring $(R,{\mathfrak m})$ with $\mathfrak p\in \Spec(R)$ we denote $E_R(R/\mathfrak p)$ by the $R$-injective hull of $R/\mathfrak p$. We will show that it has an $\hat{R}^\mathfrak p$-module structure and there is an isomorphism $E_R(R/\mathfrak p)\cong E_{\hat{R}^\mathfrak p}(\hat{R}^\mathfrak p/\mathfrak p\hat{R}^\mathfrak p)$ where $\hat{R}^\mathfrak p$ stands for the $\mathfrak p$-adic completion of $R$. Moreover for a complete Cohen-Macaulay ring $R$ the module $D(E_R(R/\mathfrak p))$ is isomorphic to $\hat{R}_\mathfrak{p}$ provided that $\dim(R/\mathfrak p)=1$ and $D(\cdot)$ denotes the Matlis dual functor $\Hom_R(\cdot, E_R(R/\mathfrak m))$. Here $\hat{R}_\mathfrak{p}$ denotes the completion of ${R_\mathfrak p}$  with respect to the maximal ideal $\mathfrak pR_\mathfrak p$. These results extend those of Matlis (see \cite{m}) shown in the case of the maximal ideal ${\mathfrak m}$.
\end{abstract}
\maketitle
\section{Introduction}
Throughout this paper $R$ is a Noetherian local ring with the maximal ideal $\mathfrak m$ and the residue field $k=R/\mathfrak{m}$. We denote the (contravaraint) Hom-functor $\Hom_R(\cdot, E_R(R/\mathfrak p))$ by $\vee$, that is $M^\vee: =\Hom_R(M, E_R(R/\mathfrak p))$ for an $R$-module $M$ and $E_R(R/\mathfrak p)$ is a fixed $R$-injective hull of $R/\mathfrak p$ where $\mathfrak p\in \Spec(R)$. Moreover we denote the Matlis dual functor by $D(\cdot):= \Hom_R(\cdot, E_R(k))$. Also $\hat{R}_\mathfrak{p}$ (resp. ${\hat{R}}^\mathfrak p$) stands for the completion of ${R_\mathfrak p}$ (resp. of $R$) with respect to the maximal ideal $\mathfrak pR_\mathfrak p$ (resp. with respect to the prime ideal $\mathfrak{p}$).

Our main goal is to give the structure of $E_R(R/\mathfrak p)$ as an ${\hat{R}}^\mathfrak p$-module. In particular we will show the following result:

\begin{Theorem}
Let $(R,\mathfrak m)$ be a ring and $\mathfrak p\in \Spec(R)$. Then $E_R(R/\mathfrak p)$ admits the structure of an ${\hat{R}}^\mathfrak p$-module and $E_R(R/\mathfrak p)\cong E_{{\hat{R}}^\mathfrak p}(\hat{R}^\mathfrak p/\mathfrak p\hat{R}^\mathfrak p)$ as ${\hat{R}}^\mathfrak p$-modules.
\end{Theorem}

In the case of the maximal ideal $\mathfrak m$, Matlis has shown (see \cite{m}) that the injective hull $E_R(k)$ of the residue field $k$ has the structure of ${\hat{R}}$-module and it is isomorphic to the ${\hat{R}}$-module $E_{\hat{R}}(k)$. Here we will extend this result to an arbitrary prime ideal $\mathfrak p$, that is $E_R(R/\mathfrak p)$ admits the structure of an ${\hat{R}}^\mathfrak p$-module and it is isomorphic to the ${\hat{R}}^\mathfrak p$-module $E_{{\hat{R}}^\mathfrak p}(\hat{R}^\mathfrak p/\mathfrak p\hat{R}^\mathfrak p)$.

Moreover, we know that $\Hom_R(E_R(R/\mathfrak p), E_R(R/\mathfrak q))= 0$ for any $\mathfrak p,\mathfrak q \in \Spec(R)$ such that $\mathfrak p\in V(\mathfrak q)$ and $\mathfrak p\neq \mathfrak q$. It was proven by Enochs (see \cite[p. 183]{e1}) that the module $E_R(R/\mathfrak q)^\vee$ has the following decomposition:
\[
E_R(R/\mathfrak q)^\vee\cong\prod T_{\mathfrak p'}.
\]
where $\mathfrak p'\in \Spec(R)$ and $T_{\mathfrak p'}$ denotes the completion of a free $R_{\mathfrak p'}$-module with respect to $\mathfrak p'R_{\mathfrak p'}$-adic completion. Here we are succeeded to prove that:
\begin{Theorem}
Let $(R,\mathfrak m)$ be a complete Cohen-Macaulay ring of dimension $n$. Suppose that $\mathfrak p\in \Spec(R)$ with $\dim(R/\mathfrak p)=1$. Then there is an isomorphism
\[
D(E_R(R/\mathfrak p))\cong  \hat{R}_\mathfrak{p}.
\]
\end{Theorem}
Recently Schenzel has shown (see \cite[Theorem 1.1]{p1}) that if $\mathfrak p$ is a one dimensional prime ideal then $D(E_R(R/\mathfrak p))\cong  \hat{R}_\mathfrak{p}$ if and only if $R/\mathfrak p$ is complete. 

Furthermore Enochs (see \cite[p. 183]{e1}) has shown that the module $E_R(R/\mathfrak q)^\vee$ is a flat cover of some cotorsion module. In the end of the section 2 we will show the following result:

\begin{Theorem}\label{12}
Let $(R,\mathfrak m)$ be a ring and $\mathfrak p,\mathfrak q \in \Spec(R)$ such that $\mathfrak p\in V(\mathfrak q)$. Then $E_R(R/\mathfrak q)^\vee$ is a flat precover of $(R/\mathfrak q)^\vee.$
\end{Theorem}

\section{Preliminaries}
In this section we will fix the notation of the paper and summarize a few preliminaries and auxiliary results. Notice that the following two Propositions are known for the maximal ideal $\mathfrak m$ we will extend them to any prime ideal $\mathfrak p\in \Spec(R)$.

\begin{Proposition}\label{1.1}
Let $(R,\mathfrak m)$ be a ring and $\mathfrak p\in \Spec(R)$. Then we have:
\begin{itemize}
\item[(a)] $E_R(R/\mathfrak p)\cong E_{R_\mathfrak q}(R_\mathfrak q/\mathfrak pR_\mathfrak q)$ for all $\mathfrak q\in V(\mathfrak p)$.
\item[(b)] Suppose that $M$ is an $R$-module and $\mathfrak p\in\Supp_R(M)$ then $M^\vee$ has a structure of an $R_\mathfrak p$-module.
\end{itemize}
\end{Proposition}
\begin{proof}
For the proof of the statement $(a)$ we refer to \cite[Theorem 18.4]{mat}. We only prove the last claim. For this purpose note that $E_R(R/\mathfrak p)$ is an $R_\mathfrak p$-module. Now we define $(\frac{r}{s}\cdot f)(m):= \frac{r}{s}f(m)$ where $\frac{r}{s}\in R_\mathfrak p, m\in M$ and $f\in M^\vee$. Then it defines the structure of $M^\vee$ to be an $R_\mathfrak p$-module which completes the proof of Proposition.
\end{proof}
\begin{Proposition}\label{4}
Let $(R,\mathfrak m)$, $(S,\mathfrak n)$  be local rings and $R\rightarrow S$ be a surjective local homomorphism $i.e.$ $ S= R/I$ for some ideal $I\subseteq R$. Suppose that $M$ is an $R$-module and $\mathfrak p\in \Spec(R)$. Then the following are true:
\begin{itemize}
\item[(a)] Suppose that $N$ is an $S$-module then for any $\mathfrak p\in V(I)$
\[
N^\vee\cong \Hom_S(N, E_S(S/\mathfrak pS)).
\]
\item[(b)] For all $\mathfrak q\in V(\mathfrak p)$ we have
\[
M^\vee\cong \Hom_{R\mathfrak q}(M_\mathfrak q, E_{R_\mathfrak q}(R_\mathfrak q/\mathfrak pR_\mathfrak q)).
\]
\end{itemize}
\end{Proposition}
\begin{proof}
For the proof see \cite[Example 3.6 and Excercise 13]{h}.
\end{proof}
\begin{Lemma}\label{4.5}
For a local ring $R$ let $M,N$ be any $R$-modules. Then for all $i\in \mathbb{Z}$ the following hold:
\begin{itemize}
\item[(1)] $\Ext^{i}_R(N,D(M))\cong D(\Tor_{i}^R(N, M))$.

\item[(2)] If in addition $N$ is finitely generated then
\[
D(\Ext^{i}_R(N,M))\cong \Tor_{i}^R(N, D(M)).
\]
\item[(3)] $\Hom_R(\varinjlim M_n, M)\cong \lim\limits_{\longleftarrow} \Hom_R(M_n, M)$
\end{itemize}
where $\{M_n: n\in \mathbb{N}\}$ is a direct system of $R$-modules.
\end{Lemma}
\begin{proof}
For the proof see \cite[Example 3.6]{h} and \cite{w}.
\end{proof}
In order to prove the next results we need a few more preparations. Let $I\subseteq R$ be an ideal and for $i= 1,\dots,s$ let $\mathfrak{q_i}$ belongs to a minimal primary decomposition of the zero ideal in $R$ where $\Rad (\mathfrak{q_i})= \mathfrak{p_i}$ . Then we denote
$u(I)$ by the intersection of all $\mathfrak{p_i}$-primary
components $\mathfrak{q_i}$ such that $\dim_R(R/(\mathfrak{p_i}+I))> 0$. Moreover we denote the functor of the global transform by $T(\cdot):=  \lim\limits_{\longrightarrow} \Hom_R(\mathfrak{m}^\alpha, \cdot)$. Also note that the local cohomology functor with respect to $I$ is denoted by $H^i_I(\cdot), i \in \mathbb Z,$ see \cite{b} for its definition and applications in local algebra.
\begin{Lemma}\label{45.111}
Let $R$ be a local ring and $I\subseteq R$ be an ideal. Then there is an exact sequence:
\[
0\to \hat{R}^I/u(I\hat{R}^I)\to \lim_{\longleftarrow} T(R/I^s)\to \lim_{\longleftarrow} H^1_\mathfrak{m} (R/I^s)\to 0
\]
where $\hat{R}^I$ denotes the $I$-adic completion of $R$.
\end{Lemma}
\begin{proof}
Note that for each $\alpha\in \mathbb{N}$ there is an exact sequence
\[
0\to \mathfrak{m}^\alpha\to R\to R/\mathfrak{m}^\alpha\to 0.
\]
For $s\in \mathbb{N}$ apply the functor $\Hom_R(\cdot, R/I^s)$ to this sequence. Then it induces the following exact sequence
\[
0\to \Hom_R(R/\mathfrak{m}^\alpha, R/I^s)\to R/I^s\to \Hom_R(\mathfrak{m}^\alpha, R/I^s)\to \Ext^1_R(R/\mathfrak{m}^\alpha, R/I^s)\to 0
\]
Now take the direct limit of this we again get an exact sequence
\[
0\to H^{0}_\mathfrak{m}(R/I^s)\to R/I^s\to T(R/I^s)\to H^{1}_\mathfrak{m}(R/I^s)\to 0
\]
for each $s\in \mathbb{N}$. It induces the following two short exact sequences
\[
0\to H^{0}_\mathfrak{m}(R/I^s)\to R/I^s\to R/I^s:\mathfrak{m}\to 0\text { and }
\]
\[
0\to R/I^s:\mathfrak{m}\to T(R/I^s)\to H^{1}_\mathfrak{m}(R/I^s)\to 0.
\]
Note that the inverse systems at the left hand side of these short exact sequences satisfy the Mittag-Leffler condition. So if we take the inverse limit to them then the resulting sequences will be exact also (see \cite[Proposition 10.2]{a}). If we combined these resulting sequences we get the following exact sequence
\[
0\to \lim_{\longleftarrow}H^{0}_\mathfrak{m}(R/I^s)\to \hat{R}^I\to \lim_{\longleftarrow}T(R/I^s)\to \lim_{\longleftarrow}H^{1}_\mathfrak{m}(R/I^s)\to 0.
\]
But by \cite[Lemma 4.1]{p5} there is an isomorphism $\lim\limits_{\longleftarrow}H^{0}_\mathfrak{m}(R/I^s)\cong  u(I\hat{R}^I).$ Then the exactness of the last sequence provides the required statement.
\end{proof}
In the next context we need the definition of the canonical module. To do this we recall the Local Duality Theorem (see \cite{goth}). Let $(S,{\mathfrak n})$ be a local Gorenstein ring of dimension $t$ and $N$ be a finitely generated $R$-module where $R= S/I$ for some ideal $I\subseteq S$. Then It was shown by Grothendieck (see \cite{goth}) that there is an isomorphism
\[
H^i_{\mathfrak m} (N)\cong \Hom_R(\Ext^{t-i}_S(N, S), E)
\]
for all $i\in \mathbb N$ (see \cite{goth}). For a slight extension of the Local Duality to an arbitrary module $M$ over a Cohen-macaulay local ring see \cite[Lemma 3.1]{waq1}. Note that a more general result was proved by Hellus (see \cite[Theorem 6.4.1]{he}). Now we are able to define the canonical module as follows:

\begin{Definition}
With the notation of the above Local Duality Theorem we define
\[
K_N:= \Ext^{t-r}_S(N, S), \dim(N)= r
\]
as the canonical module of $N$. It was introduced by Schenzel (see \cite{p4}) as
the generalization of the canonical module of a Cohen-Macaulay ring (see e.g. \cite{her}).
\end{Definition}
\begin{Corollary}\label{45}
With the notation of Lemma \ref{45.111} suppose that $I$ is a one dimensional ideal. Then the following are true:
\begin{itemize}
\item[(a)] There is an exact sequence
\[
0\to \hat{R}^I/u(I\hat{R}^I)\to \oplus_{i=1}^{s} \hat{R}_\mathfrak{p_i}\to \lim_{\longleftarrow} H^1_\mathfrak{m} (R/I^s)\to 0
\]
where $\mathfrak{p_i}\in \Ass_R(R/I)$ such that $\dim (R/\mathfrak{p_i})= 1$ for $i= 1,\dots, s.$
\item[(b)] Suppose in addition that $R$ is complete Cohen-Macaulay then there is an exact sequence
\[
0\to R/u(I)\to \oplus_{i=1}^{s} \hat{R}_\mathfrak{p_i}\to D(H^{n-1}_I(K_R))\to 0.
\]
\end{itemize}
\end{Corollary}
\begin{proof}
Note that for the proof of the statement $(a)$ it will be enough to prove the following isomorphism
\[
\lim_{\longleftarrow}T(R/I^s)\cong \oplus_{i=1}^{s} \hat{R}_\mathfrak{p_i}
\]
(by Lemma \ref{45.111}). Since $\dim (R/I)= 1$ it implies that there exists an element $x\in \mathfrak{m}$ such that $x$ is a parameter of $R/I^s$ for all $s\in \mathbb{N}$. Then it induces the following isomorphism
\[
T(R/I^s)\cong R_x/I^sR_x \text { for all $s\in \mathbb{N}$}.
\]
Now suppose that $S= \cap_{i=1}^{s} (R\setminus \mathfrak{p_i})$ then $x\in S$ and by Local Global Principal for each $s\in \mathbb{N}$ there is an isomorphism
\[
R_x/I^sR_x\cong R_S/I^sR_S.
\]
Since $R_S$ is a semi local ring so there is an isomorphism
\[
R_S/I^sR_S\cong \oplus_{i=1}^{s} R_\mathfrak{p_i}/I^sR_\mathfrak{p_i} \text { for all $s\in \mathbb{N}$}.
\]
(by Chinese Remainder Theorem). Since $\dim (R/\mathfrak{p_i})= 1$ for $i= 1,\dots, s$ so it follows that $\Rad (IR_\mathfrak{p_i})= \mathfrak{p_i}R_\mathfrak{p_i}$. By passing to the inverse limit of the last isomorphism we have
\[
\lim_{\longleftarrow}T(R/I^s)\cong \oplus_{i=1}^{s} \hat{R}_\mathfrak{p_i}
\]
which is the required isomorphism. To prove $(b)$ note that the canonical module of $K_R$ of $R$ exists (since $R$ is complete Cohen-Macaulay). So by \cite[Lemma 3.1]{waq1} there is an isomorphism
\[
 D(H^i_{\mathfrak m} (R/I^s))\cong \Ext^{n-i}_R(R/I^s, K_R)
\]
for each $i\in \mathbb N$. Since $H^i_{\mathfrak m} (R/I^s)$ is an Artinian $R$-module and $R$ is complete so by Matlis Duality its double Matlis dual is itself. That is for each $s\in \mathbb N$ there is an isomorphism
\[
H^i_{\mathfrak m} (R/I^s)\cong D(\Ext^{n-i}_R(R/I^s, K_R)).
\]
By passing to the inverse limit of this isomorphism induces the following isomorphism
\[
\lim_{\longleftarrow} H^i_\mathfrak{m} (R/I^s)\cong \lim_{\longleftarrow} D(\Ext^{n-i}_R(R/I^s, K_R)).
\]
Now take the Matlis dual of the isomorphism $H^{n-i}_I(K_R)\cong \lim\limits_{\longrightarrow} \Ext^{n-i}_R(R/I^s, K_R).$ It induces that
\[
\lim_{\longleftarrow} H^i_\mathfrak{m} (R/I^s)\cong D(H^{n-i}_I(K_R))
\]
(by Lemma \ref{4.5} $(3)$). For $i= 1$ and by the exact sequence in $(a)$ we can get the exact sequence of $(b)$ (since $R$ is complete). This finishes the proof of the Corollary.
\end{proof}
\section{On Matlis Duality}

In this section we will prove one of the main result. Actually this result was proved by Matlis for the maximal ideal $\mathfrak m$. Here we will extend this to any prime ideal $\mathfrak p\in \Spec(R)$.
\begin{Theorem}\label{2}
Let $(R,\mathfrak m)$ be a ring and $\mathfrak p\in \Spec(R)$. Then $E_R(R/\mathfrak p)$ admits the structure of an ${\hat{R}}^\mathfrak p$-module and $E_R(R/\mathfrak p)\cong E_{{\hat{R}}^\mathfrak p}(\hat{R}^\mathfrak p/\mathfrak p\hat{R}^\mathfrak p).$
\end{Theorem}
\begin{proof}
If $x\in E_R(R/\mathfrak p)$ then by \cite{m} $\Ass_R(Rx)= \{\mathfrak p\}$. It follows that some power of $\mathfrak p$ annihilates $x$.

Now let $\hat{r}\in \hat{R}^\mathfrak p$ then $\hat{r}= (r_1+\mathfrak p,\ldots,r_n+\mathfrak p^n,\ldots)$ such that $r_{n+1}- r_n\in \mathfrak p^n$ where $r_i+\mathfrak p\in R/\mathfrak p^i$ for all $i\geq 1.$ Since $x\in E_R(R/\mathfrak p)$ then there exists a fixed $n\in \mathbb N$ such that $\mathfrak p^nx= 0$ then choose $r\in R$ such that $\hat{r}- r\in \mathfrak p^n$ (by definition of the completion). Define $\hat{r}x= rx$ then it is clear that this gives a well-defined $\hat{R}^\mathfrak p$-module structure to $E_R(R/\mathfrak p)$ which agrees with its $R$-module structure. Since $E_R(R/\mathfrak p)$ is an essential extension of $R/\mathfrak p$ as an $R$-module then it necessarily is also an essential extension of $\hat{R}^\mathfrak p/\mathfrak p\hat{R}^\mathfrak p$ as an $\hat{R}^\mathfrak p$-module. So $E_R(R/\mathfrak p)\subseteq E_{{\hat{R}}^\mathfrak p}(\hat{R}^\mathfrak p/\mathfrak p\hat{R}^\mathfrak p)$.

To show that $E_R(R/\mathfrak p)$ is an injective hull of $\hat{R}^\mathfrak p/\mathfrak p\hat{R}^\mathfrak p$ as an $\hat{R}^\mathfrak p$-module it is enough to prove that $E_R(R/\mathfrak p)= E_{{\hat{R}}^\mathfrak p}(\hat{R}^\mathfrak p/\mathfrak p\hat{R}^\mathfrak p)$. To do this it suffices to see that $E_{{\hat{R}}^\mathfrak p}(\hat{R}^\mathfrak p/\mathfrak p\hat{R}^\mathfrak p)$ is an essential extension of $R/\mathfrak p$ as an $R$-module since $E_R(R/\mathfrak p)$ is an extension of $R/\mathfrak p$ as an $R$-module (see \cite{m}).

Let $x\in E_{{\hat{R}}^\mathfrak p}(\hat{R}^\mathfrak p/\mathfrak p\hat{R}^\mathfrak p)$ and choose an element $\hat{r}\in \hat{R}^\mathfrak p$ such that $0\neq \hat{r}x\in \hat{R}^\mathfrak p/\mathfrak p\hat{R}^\mathfrak p$. By the argument used in the beginning of the proof applied to $\hat{R}^\mathfrak p$, $E_{{\hat{R}}^\mathfrak p}(\hat{R}^\mathfrak p/\mathfrak p\hat{R}^\mathfrak p)$ and $x$ there is an $n\in \mathbb N$ such that $\mathfrak p^nx= 0$. Choose $r\in R$ such that $\hat{r}- r\in \mathfrak p^n$ then $\hat{r}x= rx\in R/\mathfrak p\cong \hat{R}^\mathfrak p/\mathfrak p\hat{R}^\mathfrak p$ (see \cite{a}) and $rx\neq 0$. Hence $E_{{\hat{R}}^\mathfrak p}(\hat{R}^\mathfrak p/\mathfrak p\hat{R}^\mathfrak p)$ is an essential extension of $R/\mathfrak p$ as an $R$-module and $E_R(R/\mathfrak p)= E_{{\hat{R}}^\mathfrak p}(\hat{R}^\mathfrak p/\mathfrak p\hat{R}^\mathfrak p)$ which completes the proof.
\end{proof}

\begin{Remark}
$(1)$
If $\mathfrak p\in \Spec(R)$ then by the last Theorem \ref{2} for any finitely generated $R$-module $M$ we have
\[
M^\vee\cong \Hom_{{R}}({M}, \Hom_{\hat{R}^\mathfrak p}(\hat{R}^\mathfrak p,E_{{\hat{R}}^\mathfrak p}(\hat{R}^\mathfrak p/\mathfrak p\hat{R}^\mathfrak p))).
\]
By Lemma \ref{4.5} the later module is isomorphic to $\Hom_{\hat{R}^\mathfrak p}(\hat{M}^\mathfrak p, E_{{\hat{R}}^\mathfrak p}(\hat{R}^\mathfrak p/\mathfrak p\hat{R}^\mathfrak p))$. Here we use that $\hat{M}^\mathfrak p\cong M\otimes_R \hat{R}^\mathfrak p$. By Proposition \ref{1.1} it implies that $M^\vee$ has an $\hat{R}^\mathfrak p$-module structure.

$(2)$
Since $\Supp_R(E_R(R/\mathfrak p))= V(\mathfrak p)$ where $\mathfrak p\in \Spec(R)$. So from \cite[Remark A.30]{e} the natural homomorphism
\[
E_R(R/\mathfrak p)\to E_R(R/\mathfrak p)\otimes_R \hat{R}^\mathfrak p .
\]
is an isomorphism. But $E_R(R/\mathfrak p)$ is isomorphic to the module $E_{{\hat{R}}^\mathfrak p}(\hat{R}^\mathfrak p/\mathfrak p\hat{R}^\mathfrak p)$ and it admits the structure of an ${\hat{R}}^\mathfrak p$-module (see Theorem \ref{2}). Moreover the above natural homomorphism is compatible with this structure. It implies that the $\mathfrak p$-adic completion of $R$ commutes with the injective hull of $R/\mathfrak p.$
\end{Remark}
It is well known that $\Hom_R(E_R(R/\mathfrak p), E_R(R/\mathfrak q))= 0$ for any $\mathfrak p\in V(\mathfrak q)$ and $\mathfrak p\neq \mathfrak q$. Moreover Enochs has proved that
\[
E_R(R/\mathfrak q)^\vee\cong\prod T_{\mathfrak p'}.
\]
where $\mathfrak p'\in \Spec(R)$ and $T_{\mathfrak p'}$ denotes the completion of a free $R_{\mathfrak p'}$-module with respect to $\mathfrak p'R_{\mathfrak p'}$-adic completion. Here we will show that

\begin{Theorem}\label{46}
Let $(R,\mathfrak m)$ be a complete Cohen-Macaulay ring and $\mathfrak q\in \Spec (R)$ be a one dimensional prime ideal. Then there is an isomorphism
\[
D(E_R(R/\mathfrak q))\cong  \hat{R}_\mathfrak{q}.
\]
\end{Theorem}
\begin{proof}
Since $R$ is complete Cohen-Macaulay so $K_R$ exists and we have
\[
H^{i}_\mathfrak q(K_R)= 0\text { for all $i< \height (\mathfrak{q})= n-1$.}
\]
This is true because of $K_R$ is a maximal Cohen-Macaulay $R$-module of finite injective dimension and $\Supp_R(K_R)= \Spec (R)$. Let $E^{\cdot}_R(K_R)$ be a minimal injective resolution of $K_R$. Then by \cite[Theorem 1.1]{f} we have
\[
E^{\cdot}_R(R)^i\cong \bigoplus_{\height (\mathfrak p)= i} E_R(R/\mathfrak p).
\]
Moreover $\Gamma_{\mathfrak{q}}(E_R(R/{\mathfrak p}))=0$ for all $\mathfrak p\notin V(\mathfrak{q})$ and $\Gamma_{\mathfrak{q}}(E_R(R/{\mathfrak p}))= E_R(R/{\mathfrak p})$ for all $\mathfrak p\in V(\mathfrak{q})$. Then apply $\Gamma_\mathfrak{q}$ to $E^{\cdot}_R(K_R)$ it induces the following exact sequence
\[
0\to H^{n-1}_\mathfrak{q}(K_R)\to E_R(R/\mathfrak q)\to E_R(k)\to H^{n}_\mathfrak{q}(K_R)\to 0.
\]
Applying Matlis dual to this sequence yields the following exact sequence
\begin{equation}\label{ra1}
0\to D(H^{n}_\mathfrak{q}(K_R))\to R\to D(E_R(R/\mathfrak q))\to D(H^{n-1}_\mathfrak{q}(K_R))\to 0.
\end{equation}
Here we use that $D(E_R(k))\cong R$ (since $R$ is complete). By the proof Corollary \ref{45} $(b)$ and \cite[Lemma 4.1]{p5} there are isomorphisms
\[
D(H^{n}_\mathfrak{q}(K_R))\cong \lim_{\longleftarrow} H^0_\mathfrak{m} (R/\mathfrak q^s)\cong u(\mathfrak q).
\]
So the exact sequence \ref{ra1} provides the following exact sequence
\begin{equation}\label{ra11}
0\to R/u(\mathfrak q)\to D(E_R(R/\mathfrak q))\to D(H^{n-1}_\mathfrak{q}(K_R))\to 0.
\end{equation}
Now the module $D(E_R(R/\mathfrak q))$ is an $R_\mathfrak{q}$-module so there is a natural homomorphism $R_\mathfrak{q}\to D(E_R(R/\mathfrak q))$. Then tensoring with $R_\mathfrak{q}/\mathfrak{q}^sR_\mathfrak{q}$ to this homomorphism induces the following homomorphism
\[
R_\mathfrak{q}/\mathfrak{q}^sR_\mathfrak{q}\to D(E_R(R/\mathfrak q))/\mathfrak{q}^sD(E_R(R/\mathfrak q))
\]
for each $s\in \mathbb{N}$. Now take the inverse limit of it we get that
\[
\hat{R}_\mathfrak{q}\to \lim_{\longleftarrow}D(E_R(R/\mathfrak q))/\mathfrak{q}^sD(E_R(R/\mathfrak q)).
\]
On the other side since $\Supp_R (E_R(R/\mathfrak q))= V(\mathfrak{q})$ so the module $E_R(R/\mathfrak q)$ is isomorphic to $ \lim\limits_{\longrightarrow}\Hom_R(R/\mathfrak{q}^s, E_R(R/\mathfrak q))$. Then by Proposition \ref{4.5} $(3)$ there is an isomorphism
\[
D(E_R(R/\mathfrak q))\cong  \lim_{\longleftarrow}D(\Hom_R(R/\mathfrak{q}^s, E_R(R/\mathfrak q))).
\]
But again Proposition \ref{4.5} $(2)$ implies that the module $D(\Hom_R(R/\mathfrak{q}^s, E_R(R/\mathfrak q)))$ is isomorphic to $D(E_R(R/\mathfrak q))/\mathfrak{q}^sD(E_R(R/\mathfrak q))$. Therefore there is a natural homomorphism
\[
\hat{R_\mathfrak{q}}\to D(E_R(R/\mathfrak q)).
\]
Since $R$ is complete Cohen-Macaulay so by Corollary \ref{45} $(c)$ there is an exact sequence
\[
0\to R/u(\mathfrak{q})\to \hat{R}_\mathfrak{q}\to D(H^{n-1}_\mathfrak{q}(K_R))\to 0.
\]
Then this sequence together with the sequence \ref{ra11} induces the following commutative diagram with exact rows
\[
\begin{array}{cccccccc}
 0 & \to & R/u(\mathfrak{q}) & \to & \hat{R}_\mathfrak{q} & \to & D(H^{n-1}_\mathfrak{q}(K_R)) & \to 0 \\
    &   & ||&  & \downarrow &   & ||   &\\
 0 & \to & R/u(\mathfrak{q}) & \to & D(E_R(R/\mathfrak q)) & \to & D(H^{n-1}_\mathfrak{q}(K_R)) &\to 0
\end{array}
\]
Then by Five Lemma there is an isomorphism $D(E_R(R/\mathfrak q))\cong \hat{R}_\mathfrak{q}$ which is the required isomorphism.
\end{proof}
\begin{Corollary}\label{47bb}
Fix the notation of Theorem \ref{46} then we have:
\begin{itemize}
 \item [(a)] There is an exact sequence
\begin{gather*}
0\to R/u(\mathfrak q)\to D(E_{R}(R/\mathfrak q))\to
D(H^{n-1}_{\mathfrak{q}}(K(R)))\to 0.
\end{gather*}
 \item [(b)] Suppose in addition that $R$ is domain. Then there is an isomorphism
\[
D(H^{n-1}_{\mathfrak{q}}(K(R)))\cong \hat{R}_\mathfrak{q}/R.
\]
\end{itemize}
\end{Corollary}
\begin{proof}
Since $R$ is complete so apply Corollary \ref{45} $(b)$  for $I= \mathfrak{q}$ then it implies that there is an exact sequence
\[
0\to R/u(\mathfrak q)\to \hat{R}_\mathfrak{q}\to
D(H^{n-1}_{\mathfrak{q}}(K(R)))\to 0.
\]
Then the statement $(a)$ is easily follows from Theorem \ref{46}. Note that the statement $(b)$ follows from the above short exact sequence. Recall that $u(\mathfrak{q})= 0$ since $R$ is a domain.
\end{proof}
In the next context we need the following definition of flat covers.
\begin{Definition}
Let $M$ be an $R$-module and $F$ be any flat $R$-module then the linear map $\phi: F\to M$ is called a flat cover of $M$ if the following conditions hold:
\begin{itemize}
\item[(i)] For any flat $R$-module $G$ the following sequence is exact
\[
\Hom_R(G,F)\to \Hom_R(G, M)\to 0
\]
\item[(ii)] If $\phi= \phi \circ f$ for some $f\in \Hom_R(F,F)$ then $f$ is an automorphism of $F.$
\end{itemize}
\end{Definition}
Note that if only condition $(i)$ holds then $F$ is called a flat precover. Enochs proved in his paper (see \cite[Theorem 3.1]{e2}) that if $M$ has a flat precover then it also admits a flat cover and it is unique up to isomorphisms.

The following Lemma is an easy consequence of chasing diagram ( see \cite[Lemma 1.1]{e1}).
\begin{Lemma}\label{11}
Let $M$ be an $R$-module and $F, F'$ are any flat $R$-modules then we have
\begin{itemize}
\item[(a)] If $\phi': F'\to M$ is a flat precover of $M$ and $\phi: F\to M$ is a flat cover of $M$ such that $\phi'= \phi \circ f$ for some $f\in \Hom_R(F',F)$ then $f$ is surjective and $\ker(f)$ is a direct summand of $F'.$

\item[(b)] If $\phi: F\to M$ is a flat precover of $M$ then it is a cover if and only if $\ker(\phi)$ contains no non-zero direct summand of $F$.
\end{itemize}
\end{Lemma}
Enochs (see \cite[p. 183]{e1}) has shown that the module $\Hom_R(E_R(R/\mathfrak q), E_R(R/\mathfrak p))$ is a flat cover of some cotorsion module. Here we will show the following result:
\begin{Theorem}\label{12}
Let $(R,\mathfrak m)$ be a ring and $\mathfrak p,\mathfrak q \in \Spec(R)$ such that $\mathfrak p\in V(\mathfrak q)$. Then $E_R(R/\mathfrak q)^\vee$ is a flat precover of $(R/\mathfrak q)^\vee$.
\end{Theorem}
\begin{proof}
Note that $E_R(R/\mathfrak q)$ is as essential extension of $R/\mathfrak q$. Let $F$ be any flat $R$-module then the inclusion map $R/\mathfrak q\otimes_R F\hookrightarrow E_R(R/\mathfrak q)\otimes_R F$ induces the following exact sequence
\[
\Hom_R(E_R(R/\mathfrak q)\otimes_R F, E_R(R/\mathfrak p))\to \Hom_R(R/\mathfrak q\otimes_R F, E_R(R/\mathfrak p))\to 0
\]
By the adjunction formula (see Lemma \ref{4.5}) we conclude that the following homomorphism is surjective for any flat $R$-module $F$
\[
\Hom_R(F, E_R(R/\mathfrak q)^\vee)\to \Hom_R(F, (R/\mathfrak q)^\vee)
\]
which proves that $E_R(R/\mathfrak q)^\vee$ is a flat precover of $(R/\mathfrak q)^\vee$.
\end{proof}
\begin{Remark}
Note that if $R$ is a complete local ring then $R$ is a flat cover of the residue field $k$ (see \cite[Example 5.3.19]{ee}).
\end{Remark}

\begin{Problem}
$(1)$ Let $R$ be a complete local ring and $\mathfrak p\in V(\mathfrak q)$ and $\mathfrak{p}\neq \mathfrak{q}.$ It would be of some interest to see whether $E_R(R/\mathfrak q)^\vee$ is a flat cover of $(R/\mathfrak q)^\vee$ or not$?$

$(2)$ Note that it was shown in \cite[Theorem 1.1]{p1} that if $\dim(R/\mathfrak p)=1$ then $R/\mathfrak p$ is complete if and only if $D(E_R(R/\mathfrak p))\cong  \hat{R}_\mathfrak{p}$. Let $R$ be a complete local Cohen-Macaulay ring. Is it possible to generalize Theorem \ref{46} for arbitrary prime ideals $\mathfrak q\subsetneq \mathfrak p$ $?$
\end{Problem}

\noindent\textbf{Acknowledgement.} The author is grateful to the reviewer for suggestions to
improve the manuscript.

\end{document}